\documentclass{article} 
\usepackage{amssymb,amsmath,amsthm}
\usepackage{amscd}
\usepackage{hyperref}
\usepackage{epsfig}
\usepackage{amsfonts}
\usepackage{amssymb}
\usepackage{amsmath,enumerate}
\usepackage{commath}
\usepackage{euscript}
\usepackage{enumitem}
\usepackage{amscd}
\usepackage{xcolor}
\usepackage[ruled,vlined]{algorithm2e}
\usepackage[numbers,sort&compress]{natbib}

\newtheorem{theorem}{Theorem}[section]
\newtheorem{remark}{Remark}[section]

\newtheorem{lemma}{Lemma}[section]

\begin{document}
	\title{Projected solution for Generalized Nash Games with Non-ordered Preferences}
	\author{
		Asrifa Sultana\footnotemark[1] \footnotemark[2] , Shivani Valecha\footnotemark[2] }
	\date{ }
	\maketitle
	\def\thefootnote{\fnsymbol{footnote}}
	
	\footnotetext[1]{ Corresponding author. e-mail- {\tt asrifa@iitbhilai.ac.in}}
	\noindent
	\footnotetext[2]{Department of Mathematics, Indian Institute of Technology Bhilai, Raipur - 492015, India.
	}
	
	
		
	\begin{abstract}\noindent
	Any individual's preference represents his choice in the set of available options. It is said to be complete if the person can compare any pair of available options. We aim to initiate the notion of projected solutions for the generalized Nash equilibrium problem with non-ordered (not necessarily complete and transitive) preferences and non-self constraint map. We provide the necessary and sufficient conditions under which projected solutions of a quasi-variational inequality and the considered GNEP coincide. Based on this variational reformulation, we derive the occurrence of projected solutions for the considered GNEP. Alternatively, by using a fixed point result, we ensure the existence of projected solutions for the considered GNEP without requiring the compactness of choice sets.
	\end{abstract}
	\section{Introduction}
	 In the past few years, generalized Nash equilibrium problems (GNEP) have gained attention because of their applications in electricity market models, economic equilibrium problems, models for environmental sustainability, etc. (see \cite{debreu, ausselproj,cotrinatime, faccsurvey}). The concept of GNEP was initiated by Arrow-Debreu \cite{debreu} to study economic equilibrium problems. Consider a set $\Lambda=\{1,2,\cdots N\}$ consisting of $N$-players. Suppose any player $i\in \Lambda$ regulates a strategy variable $x_i\in \mathbb{R}^{n_i}$ where $\sum_{i\in \Lambda} n_i=n$. Then, we indicate $x=(x_i)_{i\in \Lambda}\in \mathbb{R}^n$ as $x=(x_{-i},x_i)$ where $x_{-i}$ is a vector formed by strategies of all the players except $i$. Suppose the feasible strategy set for player $i$ is given as $K_i(x)\subseteq \mathbb{R}^{n_i}$, which depends on his strategy and the strategies chosen by rivals. For a given strategy $x_{-i}$ of rivals, any player $i$ intends to find a strategy $x_i\in K_i(x)$ such that $x_i$ solves the following problem,
	\begin{equation}\label{GNEP2}
	u_i(x_{-i},x_i) = \max_{y_i\in K_i(x)} u_i(x_{-i}, y_i),
	\end{equation}
	where $u_i:\mathbb{R}^n\rightarrow \mathbb{R}$ denotes objective function of any player $i$ in $\Lambda$.
	Suppose $Sol_i(x_{-i})$ consists of all such vectors $x_i$, which solves the problem (\ref{GNEP2}). Then, any vector $\bar x\in \mathbb{R}^n$ is known as equilibrium for an Arrow-Debreu GNEP $\Upsilon=(X_i,K_i,u_i)_{i\in \Lambda}$ if $\bar x_i\in Sol_i(\bar x_{-i})$ for each $i\in \Lambda$ \cite{debreu,faccsurvey}.
	
	Any individual's preference represents his choice in the set of available options. A preference is said to be complete if he can compare any pair of available options. The above form of Arrow-Debreu GNEP consisting of the real-valued objective functions is well known \cite{faccsurvey}, but it is possible only if the players' preferences are transitive and complete \cite[Proposition 2.5]{kreps}. Shafer-Sonnenschein \cite{shafer} studied the generalized Nash equilibrium problem having non-ordered (that is, not necessarily complete and transitive) preferences after noticing the fact that the completeness property is not always valid for preferences in real-world scenarios. 
	Suppose non-empty convex set $X_i\subseteq \mathbb{R}^{n_i}$ denotes the choice set for each $i\in \Lambda=\{1,2,\cdots N\}$, where $\sum_{i\in \Lambda}n_i=n$. Let us indicate,
	$$X=\prod_{i\in \Lambda} X_i \subseteq \mathbb{R}^n~\text{and}~ X_{-i}= \prod_{j\in (\Lambda\setminus \{i\})} X_j \subseteq \mathbb{R}^{n-n_i}.$$ 
	Let generalized Nash game (or abstract economy) with non-ordered preference maps be indicated by $\Gamma=(X_i,K_i,P_i)_{i\in \Lambda}$ where $P_i:X\rightrightarrows X_i$ and $K_i:X\rightrightarrows X_i$ are preference map and strategy map, respectively. Then, a vector $\tilde x\in X$ is known as equilibrium \cite{shafer,tian, yannelis} for $\Gamma=(X_i,K_i,P_i)_{i\in \Lambda}$ if,
	\begin{equation}\label{GNEPshafer}
	\tilde x_i\in K_i(\tilde x)~\text{and}~ P_i(\tilde x)\cap K_i(\tilde x)=\emptyset~\text{for every}~ i\in \Lambda.
	\end{equation}
	Suppose $D_i\subseteq X_i$ for each $i\in \Lambda$ and $K_i(x)=D_i$ for all $x\in X$. Then the game $\Gamma$ reduces to Nash equilibrium problem $NEP(D_i,P_i)_{i\in \Lambda}$ with non-ordered preferences (see \cite{scalzo1}). Recently, He-Yannelis \cite{yannelis} and Scalzo \cite[Section 4.3]{scalzo2} ensured the occurrence of equilibrium for generalized Nash games with non-ordered preference maps. Furthermore, Milasi et al. \cite{milasipref,milasipref2021} studied the maximization problem for preferences (represented by binary relations) using variational approach and as an application they derived existence results for economic equilibrium problems under uncertainty.
	
	On the other hand, Aussel et. al. \cite{ausselproj} considered an Arrow-Debreu GNEP $\Upsilon=(X_i,K_i,u_i)_{i\in \Lambda}$ having non-self constraint map, that is, for given strategy $x_{-i}$ of rival players $K_i(x)$ is not necessarily subset of $X_i$ and the product map  $K=\prod_{i\in \Lambda}K_i$ need not be self. The authors in \cite{ausselproj} shown that the classical Nash equilibrium may not exist for such games as $K_i(x)\cap X_i$ is possibly empty for some $i\in \Lambda$. They have provided an illustration of deregulated electricity market model which motivated them to introduce the notion of projected solution for the GNEP with non-self constraint map. The authors demonstrated the existence of projected solution for the GNEP having Euclidean spaces as strategy spaces. Furthermore, Bueno-Cotrina \cite{cotrinaGNEP} established the existence of projected solutions for generalized Nash games having non-self constraint maps defined over Banach spaces. It is worth mentioning that one can ensure the occurrence of projected solution by employing the existing results \cite{ausselproj,cotrinaGNEP} only if the numerical representation for the preferences is available. 
	
	In this article, our aim is to study the notion of projected solution for the generalized Nash equilibrium problem $\Gamma=(X_i,K_i,P_i)_{i\in \Lambda}$ in which the preference maps $P_i$ are non-ordered and the product map $K=\prod_{i\in \Lambda}K_i$ is non-self. In this regard, we first provide the necessary and sufficient conditions under which projected solution of a quasi-variational inequality and the considered GNEP coincides. Based on this variational reformulation, we derive the occurrence of projected solution for the considered GNEP. Alternatively, we ensure the occurrence of projected solution for the considered GNEP defined over non-compact choice sets by using a fixed point result.

\section{Preliminaries}
The polar cone $C^{\circ}$ for any set $C \subset \mathbb{R}^m$ is given as, 
$$ C^{\circ}= \{x^*\in \mathbb{R}^m|\,\langle x^*,x\rangle\leq 0~\text{for all}~x\in C\}$$
and $C^\circ=\mathbb{R}^m$ if $C$ is an empty set \cite{ausselnormal,cotrina}.
Furthermore, the normal cone of the set $C$ at some point $x\in \mathbb{R}^m$ is given as $N_C(x)=(C-\{x\})^\circ$, that is, 
\begin{align}
N_C(x)=
\begin{cases}
\{x^*\in \mathbb{R}^m|\,\langle x^*,y-x\rangle \leq 0~\text{for all}~y\in C\},&\text{if}~C\neq\emptyset\\
\mathbb{R}^m,&\text{otherwise}.
\end{cases}
\end{align}

Suppose $C\subseteq \mathbb{R}^m$ is arbitrary. We indicate convex hull and closure of the set $C$ by $co(C)$ and $cl(C)$, respectively. Let $P:C\rightrightarrows \mathbb{R}^p$ be a multi-valued map, that is, $P(x)\subseteq \mathbb{R}^p$ for any $x\in C$. A map $P$ has open upper sections \cite{tian}, if $P(x)$ is open for any $x\in C$. The reader may refer to \cite{homidan,aliprantis}, in order to recall some important concepts of upper semi-continuous (u.s.c.), lower semi-continuous (l.s.c.) and closed multi-valued maps. 

For $X\subseteq \mathbb{R}^n$, define a projection map $Pr_X:\mathbb{R}^n\rightrightarrows X$ as,
\begin{equation}
Pr_X(y)=\{x\,|\,\norm{y-x}=\inf_{w\in X}\norm{y-w}\}.
\end{equation}
Let us recall a result related to these maps, which we will use to prove an upcoming existence result.
\begin{lemma}\label{projection}\cite[Lemma 2.1 and Theorem 2.3]{stampbook}
	Let $X\subset\mathbb{R}^n$ be non-empty closed convex. Then, for any $y\in \mathbb{R}^n$ there exists a unique $x\in X$ such that $Pr_X(y)=\{x\}$. Further, $Pr_X$ is a continuous map and $Pr_X(y)=\{x\}$ if and only if,
	\begin{equation*}
	\langle x-y, \eta-x\rangle \geq 0,~\text{for all}~\eta\in X.
	\end{equation*}
\end{lemma}

We employ following result to ensure the occurrence of projected solution for the GNEP with non-ordered preference maps.
\begin{theorem}\cite[Theorem 2]{himel}\label{kakutani}
	Suppose $Y\subseteq \mathbb{R}^n$ is non-empty convex and $D$ is compact subset of $Y$. Assume that the map $T:Y\rightrightarrows D$ is u.s.c. with non-empty convex closed values. Then a vector $x_\circ\in X$ exists s.t. $x_\circ\in T(x_\circ)$.
\end{theorem}
\subsection{Tools Required for Variational Reformulation}


Suppose $S^>_{f(x,y)}=\{z\in \mathbb{R}^p|\,f(x,z)> f(x,y)\}$ denotes a strict sub-level set for a real-valued function $f:\mathbb{R}^m\times \mathbb{R}^p\rightarrow\mathbb{R}$. The normal operator $\mathcal{N}^>_f:\mathbb{R}^m\times \mathbb{R}^p\rightrightarrows\mathbb{R}^p$ \cite{homidan,ausselnormal, cotrina} corresponding to $f$ is defined as, $\mathcal N^>_{f}(x,y)= (S^>_{f(x,y)}-\{y\})^\circ,$ that is,
\begin{align}\label{normalope}
\mathcal N^>_{f}(x,y)=
\begin{cases}
\{y^*\in \mathbb{R}^p|\,\langle y^*,z-y\rangle \leq 0~\forall~z\in S^>_{f(x,y)}\},&\text{if}~S^>_{f(x,y)}\neq\emptyset\\
\mathbb{R}^p,&\text{otherwise}.
\end{cases}
\end{align}
It is well established that the normal operator $\mathcal{N}^>_f$ defined in (\ref{normalope}) becomes non-empty valued and closed under suitable assumptions on the function $f$ and sub-level set $S^>_{f(x,y)}$ (see for example \cite{homidan,cotrina}).

Suppose $P_i:\mathbb{R}^{n-n_i}\times \mathbb{R}^{n_i}\rightrightarrows\mathbb{R}^{n_i}$ denotes the preference map of player $i$. We define a map $\mathcal N_{P_i}:\mathbb{R}^{n-n_i}\times \mathbb{R}^{n_i}\rightrightarrows\mathbb{R}^{n_i}$ by adapting the mentioned normal operator $\mathcal{N}^>_f$ to the context of multi-valued map $P_i$ as 
$\mathcal N_{P_i}(x_{-i},x_i)= (P_i(x_{-i},x_i)-\{x_i\})^\circ,$ that is, 
\begin{align}\label{normal}
\mathcal N_{P_i}(x_{-i},x_i)=
\begin{cases}
\{x_i^*|\,\langle x_i^*,z_i-x_i\rangle \leq 0~\forall\,z_i\in P(x_{-i},x_i)\},&\text{if}~{P_i}(x_{-i},x_i)\neq\emptyset\\
\mathbb{R}^{n_i},&\text{otherwise}.		
\end{cases}	
\end{align}

In order to study the game $\Gamma$ through variational inequality theory, 
we define a map $T_i:\mathbb{R}^{n-n_i}\times \mathbb{R}^{n_i}\rightrightarrows \mathbb{R}^{n_i}$ for any $i\in \Lambda$ as follows,
\begin{equation} \label{Ti}
T_i(x_{-i},x_i)=co(\mathcal{N}_{\tilde P_i}(x_{-i},x_i)\cap S_i[0,1]),
\end{equation} 
where $\mathcal{N}_{\tilde P_i}:\mathbb{R}^{n-n_i}\times \mathbb{R}^{n_i}\rightrightarrows \mathbb{R}^{n_i}$ is defined corresponding to $\tilde P_i(x_{-i},x_i)= co(P_i(x_{-i},x_i))$ as (\ref{normal}) and $S_i[0,1]=\{x\in \mathbb{R}^{n_i}|\,\norm{x}=1\}$. Suppose $T:\mathbb{R}^n\rightrightarrows \mathbb{R}^n$ is defined as, 
\begin{equation}
T(x)=\prod_{i\in \Lambda}T_i(x).\label{T}
\end{equation} 
Then, clearly the map $T$ is convex and compact valued map. 

We employ the following results for obtaining projected solutions of the considered GNEP through variational reformulation. 
\begin{lemma} \label{Tusc}
	Suppose $T:\mathbb{R}^{n}\rightrightarrows \mathbb{R}^{n}$ is a multi-valued map defined as (\ref{T}) and $Z\subseteq \mathbb{R}^n$ is closed. If for each $i\in \Lambda$,
	\begin{itemize}
		\item[(a)] $P_i$ is l.s.c. map on $Z$ then the map $T$ is u.s.c. over $Z$;
		\item[(b)] $P_i$ satisfies $x_i\notin co (P_i(x))$ for each $x\in Z$ then $T$ is non-empty valued on the set $Z$.
	\end{itemize}
\end{lemma}
\begin{proof}
	To prove (a), suppose $i\in \Lambda$ is arbitrary. Then, $\tilde P_i$ is l.s.c. map according to \cite[Theorem 5.9]{rockafellar}. We claim that $\mathcal{N}_{\tilde P_i}:Z\rightrightarrows \mathbb R^{n_i}$ is a closed map. Let us consider sequences $\{(x^n_{-i},x^n_i)\}_{n\in \mathbb{N}}\in Z$ and $(x_i^n)^*\in  \mathcal N_{\tilde P_i}(x^n_{-i},x^n_i)$ with $(x^n_{-i},x^n_i)\rightarrow(x_{-i},x_i)$ and $(x_i^n)^*\rightarrow x_i^*$. It is enough to prove $x_i^*\in \mathcal N_{\tilde P_i}(x_{-i},x_i)$. In fact, the claim follows trivially if $\tilde P(x_{-i},x_i)=\emptyset$. Suppose $z_i\in \tilde P_i(x_{-i},x_i)$ is arbitrary. Since $\tilde P_i$ is l.s.c. on $Z$, we obtain some sequence $z^n_i\in \tilde P_{i}(x^n_{-i},x^n_i)$ such that $z^n_i\rightarrow z_i$. The fact that $(x_i^n)^*\in  \mathcal N_{\tilde P_i}(x^n_{-i},x^n_i)$ implies,
	$$ \langle (x_i^n)^*, z_i^n-x^n_i\rangle \leq 0,\quad \forall n\in \mathbb{N}.$$ 
	Finally, we obtain $\langle x_i^*,z_i-x_i\rangle \leq 0$ by taking $n\rightarrow \infty$. Since $z_i\in \tilde P_{i}(x_{-i},x_i)$ is arbitrary, we observe $x_i^*\in \mathcal N_{\tilde P_i}(x_{-i},x_i)$. Clearly, the map $T_i:Z \rightrightarrows \mathbb{R}^{n_i}$ defined as (\ref{Ti}) becomes closed. Finally, the fact that $T_i(x)\subseteq \bar B_i(0,1)$ for any $x$ in $Z$ implies $T_i$ is u.s.c. map \cite[Theorem 17.11]{aliprantis}. This leads us to the conclusion $T=\prod_{i\in \Lambda} T_i$ meets upper semi-continuity over $Z$.
	
	To prove (b), suppose $(x_{-i},x_i)\in Z$ is arbitrary. We claim that $\mathcal N_{\tilde P_i}(x_{-i},x_i)\setminus\{0\}\neq \emptyset$. In fact, $\tilde P_i(x_{-i},x_i)=\emptyset$ implies $\mathcal N_{\tilde P_i}(x_{-i},x_i)=\mathbb {R}^{n_i}$. If $\tilde P_i(x_{-i},x_i)\neq \emptyset$, then using the fact $x_i\notin co (P_i(x))$ a vector $0\neq x_i^*\in \mathbb R^{n_i}$ is obtained (as per separation theorem \cite[Theorem 2.5]{aubin}) such that, $$\langle x_i^*,z_i\rangle \leq \langle x_i^*,x_i\rangle~\text{for all}~z_i\in \tilde P_i(x_{-i},x_i).$$ 
	This, finally leads us to the conclusion $x_i^*\in \mathcal N_{\tilde P_i}(x_{-i},x_i)\setminus \{0\}$. Clearly, $\frac{x_i^*}{\norm{x_i^*}}\in \mathcal{N}_{\tilde P_i}(x_{-i},x_i) \cap S_i[0,1]$. Hence, $T(x)=\prod_{i\in \Lambda} T_i(x)\neq \emptyset$ for any $x\in Z$.
\end{proof}

\begin{lemma}\label{relation}
	Suppose $i\in \Lambda$ is fixed and $P_i$ admits open upper sections over the set $Z\subset \mathbb{R}^n$. For $(x_{-i},x_i)\in Z$, if $x_i^*\in \mathcal N_{\tilde P_i}(x_{-i},x_i)$ and there exists $z_i\in \tilde P_i(x_{-i},x_i)$ such that $\langle x_i^*,z_i-x_i\rangle\geq 0$ then $x_i^*=0$. 
\end{lemma}
\begin{proof}
	One can equivalently prove that $\langle x_i^*,z-x_i\rangle<0$ whenever $0\neq x_i^*\in \mathcal N_{\tilde P}(x_{-i},x_i)$ and $z_i\in \tilde P_i(x_{-i},x_i)$. Clearly, $\tilde P_i(x_{-i},x_i)$ is an open set as $P_i(x_{-i},x_i)$ is open. 
	For any $z_i\in \tilde P_i(x_{-i},x_i)$, we have $z_i-x_i\in int (\tilde P(x_{-i},x_i)-\{x_i\})$. Hence, the result follows from \cite[Lemma 2.1]{cotrina}.
\end{proof}

\section{Variational Reformulation of Projected Solution for GNEP with Non-ordered Preferences}
Let $\Lambda=\{1,2,\cdots N\}$ be a set of involved agents. Suppose a non-empty convex set $X_i\subseteq \mathbb{R}^{n_i}$ denotes the choice set for each agent $i\in \Lambda$, where $\sum_{i\in \Lambda}n_i=n$. Assume that,
	$$X=\prod_{i\in \Lambda} X_i \subseteq \mathbb{R}^n~\text{and}~ X_{-i}= \prod_{j\in (\Lambda\setminus \{i\})} X_j \subseteq \mathbb{R}^{n-n_i}.$$ 
	Consider the multi-valued maps $P_i:\mathbb{R}^{n}\rightrightarrows \mathbb{R}^{n_i}$ and $K_i:X\rightrightarrows \mathbb{R}^{n_i}$ as preference map and constraint map, respectively, for any $i\in \Lambda$. Then, a vector $\tilde x\in X$ is said to be projected solution for the generalized Nash game $\Gamma=(X_i,K_i,P_i)_{i\in \Lambda}$ if there exists $\tilde y=(\tilde y_i)_{i\in \Lambda}\in \mathbb{R}^n$ such that,
	\begin{itemize}
		\item [(a)]$\norm{\tilde y-\tilde x}=\inf_{w\in X}\norm{\tilde y-w}$
		\item [(b)] $\tilde y_i$ solves $NEP(K_i(\tilde x),P_i)_{i\in \Lambda}$, that is, $\tilde y_i\in K_i(\tilde x)$ and $P_i(\tilde y)\cap K_i(\tilde x)=\emptyset$ for each $i\in \Lambda$.
	\end{itemize}

 If the preference map $P_i:X_{-i}\times X_i\rightrightarrows X_i$ is representable by a real-valued function $u_i:\mathbb{R}^n\rightarrow \mathbb{R}$ then $P_i(x_{-i},x_i)=\{y_i\in \mathbb{R}^{n_i}\,|\, u_i(x_{-i},y_i)>u_i(x_{-i},x_i)\}$ (see \cite{shafer}). Then, one can observe that the considered game $\Gamma=(X_i,K_i,P_i)_{i\in \Lambda}$ reduces to the Arrow-Debreu GNEP $\Upsilon= (X_i,K_i,u_i)_{i\in \Lambda}$ with non-self constraint map considered by Aussel et al. in \cite{ausselproj}. Furthermore, the concept of projected solution defined by us coincides with \cite{ausselproj} in this case.
\subsection{Necessary and Sufficient conditions for Variational Reformulation}\label{neccondi}
Suppose $T:K\rightrightarrows \mathbb{R}^n$ is a multi-valued map where $K\subseteq \mathbb{R}^n$. The Stampacchia VI problem $VI(T,K)$ \cite{ausselproj} corresponds to determine $\bar x\in K$ s.t.,
$$ \exists\,\bar x^*\in T(\bar x)~\text{satisfying}~\langle \bar x^*,z-\bar x\rangle\geq 0~\text{for each}~z\in K.$$

Suppose $T:D\rightrightarrows \mathbb{R}^n$ and $K:D\rightrightarrows D$ are multi-valued maps where $D\subseteq \mathbb{R}^n$. Then, the QVI problem $QVI(T,K)$ \cite{ausselproj} is to find $\bar x\in K(\bar x)$ s.t.,
$$\exists\,\bar x^*\in T(\bar x)~\text{satisfying}~\langle \bar x^*,z-\bar x\rangle\geq 0~\text{for each}~z\in K(\bar x).$$

Aussel et. al. \cite{ausselproj} initiated the concept of projected solution for QVI with non-self constraint maps. Suppose $T:\mathbb{R}^n\rightrightarrows \mathbb{R}^n$ and $K:D\rightrightarrows \mathbb{R}^n$ are multi-valued maps where $D\subseteq \mathbb{R}^n$. Then, a vector $\tilde x\in D$ is known as projected solution \cite{ausselproj} for $QVI(T,K)$ if there exists $\tilde y\in \mathbb{R}^n$ such that,
\begin{itemize}
	\item[(a)] $\norm{\tilde y-\tilde x}=\inf_{w\in X}\norm{\tilde y-w}$
	\item [(b)] $\tilde y$ solves $VI(T,K(\tilde x))$, that is, $\tilde y\in K(\tilde x)$ satisfies,
	$$ \exists\,\tilde y^*\in T(\tilde y)~\text{satisfying}~\langle \tilde y^*,z-\tilde y\rangle\geq 0~\text{for each}~z\in K(\tilde x).$$
\end{itemize} 
The authors in \cite[Lemma 4.1]{ausselproj} observed that the projected solutions of Arrow-Debreu GNEP and a quasi-variational inequality coincides. By using this relation, they derived the existence of projected solution for such GNEP in \cite[Theorem 4.2]{ausselproj}.

In the following result, we show that any projected solution of a quasi-variational inequality is projected solution for the considered generalized Nash game $\Gamma$  under suitable conditions.
\begin{theorem}\label{equivalence}
Suppose the map $T$ is defined as (\ref{T}) and the map $K:X\rightrightarrows X$ is defined as $K(x)=\prod_{i\in \Lambda} K_i(x)$. Then, any projected solution of $QVI(T,K)$ is also a projected solution for $\Gamma =(X_i,K_i,P_i)_{i\in \Lambda}$ if for each $i\in \Lambda$:
\begin{itemize}
	\item[(a)] $X_i\subset \mathbb{R}^{n_i}$ is non-empty convex closed;
	\item [(b)] $K_i:X\rightrightarrows \mathbb{R}^{n_i}$ admits non-empty values;
	\item[(c)] $P_i:\mathbb{R}^n\rightrightarrows \mathbb{R}^{n_i}$ admits open upper sections for any $y\in \prod_{i\in \Lambda}co (\overline{K_i(X)})$.
\end{itemize}
\end{theorem}
\begin{proof}
Suppose $\tilde x$ is a projected solution for $QVI(T,K)$ then there exists $\tilde y\in K(\tilde x)$ such that,
\begin{equation}\label{GNEPQVI}
	\exists\,\tilde y^*\in T(\tilde y), \langle \tilde y^*, z-\tilde y\rangle \geq 0~\text{for all}~ z\in K(\tilde x)
\end{equation}
and $Pr_X(\tilde y)=\tilde x$.

To prove that $\tilde x$ is a projected solution for $\Gamma$, it is enough to show that $\tilde y$ is an equilibrium for $NEP(K_i(\tilde x),P_i)_{i\in \Lambda}$. Let us assume that $\tilde y$ is not an equilibrium for the Nash game $NEP(K_i(\tilde x),P_i)_{i\in \Lambda}$. Then $P_i(\tilde y) \cap K_i(\tilde x)\neq \emptyset$ for some $i\in \Lambda$. Clearly, it yields $co(P_i(\tilde y))\cap  K_i(\tilde x)\neq \emptyset$ for some $i\in \Lambda$. Suppose $z_i\in co(P_i(\tilde y))\cap  K_i(\tilde x)$. Then, $z=(\tilde y_{-i}, z_i)\in K(\tilde x)$ and from (\ref{GNEPQVI}) we obtain,
\begin{equation}\label{rel}
	\langle \tilde y_i^*, z_i-\tilde y_i\rangle \geq 0.
\end{equation}
By virtue of Lemma \ref{relation}, one can observe that $\tilde y_i^*=0$ as $\tilde y_i^*\in co(\mathcal N_{\tilde P_i}(\tilde y)\cap S_i[0,1])\subseteq \mathcal N_{\tilde P_i}(\tilde y)$ fulfills (\ref{rel}). 

Again, using the fact that $\tilde y_i^*\in co(\mathcal N_{\tilde P_i}(\tilde y)\cap S_i[0,1])$, one can obtain $y_{1i}^*,\cdots, y_{ri}^*\in \mathcal N_{\tilde P_i}(\tilde y)\cap S_i[0,1]$ with $\lambda_1,\cdots, \lambda_r\in [0,1]$ satisfying $\sum_{k=1}^{r}\lambda_k=1$ and,
$$0= \tilde y_i^*=\sum_{k=1}^{r}\lambda_k y_{ki}^*.$$
Suppose we have $k_\circ\in \{1,\cdots, r\}$ with $\lambda_{k_\circ} >0$. Then,
$$ -y_{k_\circ i}^*=\sum_{k\neq k_\circ,k=1}^{r} \frac{\lambda_k}{\lambda_{k_\circ}} y^*_{ki}.$$
We get $-y_{k_\circ i}^*\in \mathcal{N}_{\tilde P_i}(\tilde y)$ as it is a convex cone. Hence, $y_{k_\circ i}^*\in \mathcal{N}_{\tilde P_i}(\tilde y)\cap -\mathcal{N}_{\tilde P_i}(\tilde y)$. However, $\tilde P_i(\tilde y)$ is an open set due to our assumptions on $P_i$ and it also contains $z_i$. Therefore, $\mathcal{N}_{\tilde P_i}(\tilde y)\cap -\mathcal{N}_{\tilde P_i}(\tilde y)=\{0\}=y_{k_\circ i}^*$ by \cite[Lemma 2.1]{cotrina}. But, this contradicts the fact that $y_{k_\circ i}^*\in S_i[0,1]$. Consequently, our hypothesis is false and $\tilde y$ becomes an equilibrium for Nash game $NEP(K_i(\tilde x),P_i)_{i\in \Lambda}$.
\end{proof}

Following result provides the suitable conditions under which projected solution for the considered generalized Nash games becomes projected solution for an associated quasi-variational inequality.
\begin{theorem}
Suppose the map $T$ is defined as (\ref{T}) and the map $K:X\rightrightarrows X$ is defined as $K(x)=\prod_{i\in \Lambda} K_i(x)$. Then, any solution of the generalized Nash game $\Gamma=(X_i,K_i,P_i)_{i\in \Lambda}$ solves quasi-variational inequality $QVI(T,K)$ if for each $i\in \Lambda$:
\begin{itemize}
	\item[(a)] $P_i:\mathbb{R}^n\rightrightarrows \mathbb{R}^{n_i}$ admits convex values for any $y\in \prod_{i\in \Lambda}co (\overline{K_i(X)})$;
	\item[(b)]  $y_i\in cl(P_i(y))$ for all $y\in \prod_{i\in \Lambda}co (\overline{K_i(X)})$ with $P_i(y)\neq \emptyset$;
	\item[(c)] $K_i:X\rightrightarrows \mathbb{R}^{n_i}$ is a non-empty, closed and convex valued map.
\end{itemize}
\end{theorem}
\begin{proof}
Suppose $\tilde x$ is projected solution for the given game $\Gamma$. Then, there exists $\tilde y\in K(\tilde x)$ such that $Pr(\tilde y)=\tilde{x}$ and $P_i(\tilde y)\cap K_i(\tilde x)=\emptyset$ for any $i\in \Lambda$. It is sufficient to show that $\tilde y$ solves $VI(T,K(\tilde x))$. In the view of assumption (a), we observe that $\tilde P_i(\tilde y)=P_i(\tilde y)$. Suppose $\Lambda_\circ \subseteq \Lambda$ consists of all $i$ for which $P_i(\tilde y)=\emptyset$. Then, we obtain $\mathcal N_{P_i}(\tilde y)=\mathbb{R}^{n_i}$ for any $i\in \Lambda_\circ$. Clearly, for each $i\in \Lambda_\circ$ we have $0=\tilde y_i^*\in T_i(\tilde y)$ such that $\langle \tilde y_i^*, z_i -\tilde y_i\rangle \geq 0$ for any $z_i\in K_i(\tilde x)$.

Consider, the case when $i\in \Lambda\setminus \Lambda_\circ$ and $P_i(\tilde y)\neq \emptyset$. Since $P_i(\tilde y)\cap K_i(\tilde x)=\emptyset$, we obtain some $0\neq y_i^*\in \mathbb{R}^{n_i}$ by using separation theorem \cite[Theorem 2.5]{aubin} such that $y_i^*$ satisfies,
\begin{equation}\label{separ}
	\sup_{z_i\in P_i(\tilde y)}\langle y_i^*, z_i\rangle \leq \inf_{w_i\in K_i(\tilde x)}\langle y_i^*,w_i\rangle.
\end{equation} 
Since $\tilde y_i\in K_i(\tilde x)$, we obtain,
\begin{equation*}
	\langle y_i^*, z_i-\tilde y_i\rangle \leq 0, \enspace\text{for all}~z_i\in P_i(\tilde x).
\end{equation*} 
Then, $0\neq y_i^*\in \mathcal{N}_{P_i}(\tilde y)$, and consequently $\tilde y_i^*=\frac{y_i^*}{\norm{y_i^*}}\in T_i(\tilde y).$

We claim that $\langle \tilde y_i^*, w_i-\tilde y_i\rangle \geq 0$ for all $w_i\in K_i(\tilde x)$. In fact, we already have $ \tilde y_i\in cl(P_i(\tilde y))$. Thus, one can obtain a sequence $\{\tilde y_i^n\}_{n\in \mathbb{N}}$ in $P_i(\tilde y)$ converging to $\tilde y_i$ by using convexity and non-emptiness of the set $P_i(\tilde y)$. According to (\ref{separ}),
$$\langle \tilde y_i^*, y_i^n\rangle \leq \inf_{w_i\in K_i(\tilde x)}\langle \tilde y_i^*,w_i\rangle,~\text{for all}~n\in \mathbb{N}.$$
By taking $n\rightarrow\infty$ one can obtain,
\begin{equation}\label{eqqvi}
	\langle \tilde y_i^*, w_i -\tilde y_i\rangle \geq 0~\text{for any}~w_i\in K_i(\tilde x).
\end{equation}  
Finally, we assume that,$$ \tilde y_i^*=
\begin{cases}
	0,~& \text{if}~i\in \Lambda_\circ\\
	\frac{y_i^*}{\norm{y_i^*}},~& \text{if}~i\in \Lambda\setminus \Lambda_\circ
\end{cases}$$ 
where $y_i^*$ is obtained in (\ref{separ}). In the view of (\ref{eqqvi}), one can observe that $\tilde y^*=(\tilde y_i^*)_{i\in \Lambda}\in T(\tilde y)$ satisfies
$ \langle \tilde y^*, w-\tilde y \rangle \geq 0~\text{for all}~w\in K(\tilde x).$
\end{proof}

\section{Existence of Projected Solution for GNEP}
In this section, we show the occurrence of projected solution for the game $\Gamma=(X_i,K_i,P_i)_{i\in \Lambda}$. In this regard, we use the variational reformulation of projected solutions derived in Subsection \ref{neccondi}. Further, we use Himmelberg fixed point result Theorem \ref{kakutani} to ensure the occurrence of projected solution for the game $\Gamma$ defined over unbounded choice sets $X_i$.

Let us denote $D=X\times \prod_{i\in \Lambda}co (\overline{K_i(X)})$. Define a map $\Phi:D\rightrightarrows \mathbb{R}^n\times \mathbb{R}^n$,
\begin{equation}\label{definePhi}
\Phi(x,y)= \{(f(x,y),y^*)|\,y^*\in T(y)\}
\end{equation}
where the function $f:D\rightarrow \mathbb{R}^n$ is defined as $f(x,y)=x-y$ and the map $T$ is defined as (\ref{T}). Furthermore, consider a map $\mathcal{K}:D\rightrightarrows D$ defined as
\begin{equation}\label{defineK}
\mathcal{K}(x,y)=X\times K(x).
\end{equation} 
Then, following result shows that any solution of $QVI(\Phi,\mathcal K)$ is a projected solution for the considered game.
\begin{lemma}\label{GNEPRem}
	Suppose the maps $\Phi$ and $\mathcal{K}$ are defined as (\ref{definePhi}) and (\ref{defineK}) respectively. Then, any solution of $QVI(\Phi, \mathcal{K})$ is a projected solution for the game $\Gamma=(X_i,K_i,P_i)_{i\in \Lambda}$ if for each $i\in \Lambda$:
	\begin{itemize}
		\item[(a)] $X_i\subset \mathbb{R}^{n_i}$ is non-empty convex closed;
		\item [(b)] $K_i:X\rightrightarrows \mathbb{R}^{n_i}$ admits non-empty values;
		\item[(c)] $P_i:\mathbb{R}^n\rightrightarrows \mathbb{R}^{n_i}$ admits open upper sections for any $y\in \prod_{i\in \Lambda}co (\overline{K_i(X)})$.
	\end{itemize}
\end{lemma}
\begin{proof}
	In the view of Lemma \ref{projection}, we observe that $\tilde x$ is a projected solution for $QVI(T,K)$ if and only if there exists $\tilde y\in K(\tilde x)$ such that $(\tilde x,\tilde y)\in X\times K(\tilde x)$ solves the following quasi-variational inequality: 
	\begin{equation}\label{newQVI}
	\exists\, \tilde y^*\in T(\tilde y)~\text{s.t.}~\langle f(\tilde x,\tilde y), \eta -\tilde x\rangle+\langle\tilde y^*,z-\tilde y\rangle\geq 0,~\forall\,(\eta,z)\in X\times K(\tilde x)
	\end{equation}
	where the function $f:\mathbb{R}^n\rightarrow\mathbb{R}^n$ is considered as $f(x,y)=x-y$. In fact, inequality (\ref{newQVI}) represents quasi-variational inequality $QVI(\Phi,\mathcal K)$. Eventually, any solution of $QVI(\Phi,\mathcal{K})$ (\ref{newQVI}) yields a projected solution for $QVI(T,K)$, which becomes a projected solution for the given game $\Gamma$ by virtue of Theorem \ref{equivalence}.
\end{proof}

Following lemma is used to show the occurrence of projected solution for the considered GNEP $\Gamma=(X_i,K_i,P_i)_{i\in \Lambda}$.
\begin{lemma}\label{lemmaupp}
	Assume that $X$ and $Y$ are non-empty subsets of  $\mathbb{R}^n$. Suppose a map $T:Y\rightrightarrows \mathbb{R}^n$ is u.s.c. with non-empty convex compact values and $f:X\times Y\rightarrow \mathbb{R}$ is a continuous function. Then, $\Phi:X\times Y\rightrightarrows X\times Y$ defined as $\Phi(x,y)=\{(f(x,y),y^*)|y^*\in T(y)\}$ is an u.s.c. map with non-empty convex compact values.
\end{lemma}
\begin{proof}
	The map $\Phi$ becomes non-empty convex compact valued as per our hypothesis on $T$. We show that $\Phi$ is u.s.c. Suppose $(x,y)\in X\times Y$ is arbitrary element and $O$ is open set containing $\Phi(x,y)$. Suppose $O_1$ and $O_2$ are open sets containing $f(x,y)$ and $T(y)$ such that $O_1\times O_2\subseteq O$.  Since $f$ is a continuous function $U_1=f^{-1}(O_1)$ is open set containing $(x,y)$. Further, we obtain open set $U_2$ containing $y$ such that $T(U_2)\subset O_2$ by using upper semi-continuity of $T$. Suppose $U'$ is arbitrary open neighbourhood of $x$ and $U=U_1\cap (U'\times U_2)$. Clearly, $U$ is an open set consisting of $(x,y)$ such that $\Phi(U)\subseteq O$.
\end{proof}
In following theorem, we ensure the occurrence of projected solution for GNEP $\Gamma=(X_i,K_i,P_i)_{i\in \Lambda}$ by using its variational reformulation derived in Theorem \ref{equivalence}.
\begin{theorem}\label{compactresult}
	Assume that for any $i\in \Lambda$:
	\begin{itemize}
		\item[(a)] $X_i\subset \mathbb{R}^{n_i}$ is non-empty compact convex;
		\item[(b)] $K_i:X\rightrightarrows \mathbb{R}^{n_i}$ is closed lower semi-continuous map with non-empty convex values for any $x\in X$;
		\item[(c)]$P_i:\mathbb{R}^{n}\rightrightarrows \mathbb{R}^{n_i}$ is lower semi-continuous map;
		\item[(d)] $P_i(y)$ is open set and $y_i\notin co(P_i(y))$ for any $y\in \prod_{i\in \Lambda}co (\overline{K_i(X)})$.
	\end{itemize}
	Then, $\Gamma=(X_i,K_i,P_i)_{i\in \Lambda}$ admits a projected solution if $K(X)$ is relatively compact.
\end{theorem}
\begin{proof}
	One can observe that the set $D=X\times \prod_{i\in \Lambda}co (\overline{K_i(X)})$ is non-empty convex compact. Suppose the maps $\Phi:D\rightrightarrows \mathbb{R}^n\times \mathbb{R}^n$ and $\mathcal{K}:D\rightrightarrows D$ are defined as (\ref{definePhi}) and (\ref{defineK}), respectively.
	
	According to Lemma \ref{GNEPRem}, it is sufficient to show that $QVI(\Phi,\mathcal{K})$ (\ref{newQVI}) admits a solution $(\tilde x,\tilde y)\in X\times K(\tilde x)$.
	We know that the map $T$ is u.s.c. with non-empty convex and compact values over the set $\prod_{i\in \Lambda}co (\overline{K_i(X)})$ as per Lemma \ref{Tusc}. Hence, $\Phi$ becomes u.s.c. map with non-empty convex compact values by virtue of Lemma \ref{lemmaupp}. In the view of hypothesis (b), $\mathcal K$ is l.s.c. closed map with non-empty convex values as per \cite[Lemma 2]{ausselradner}. By employing \cite[Theorem 3 (Corollary)]{tan}, we obtain a solution $(\tilde x,\tilde y)\in X\times K(\tilde x)$ for $QVI(\Phi,\mathcal K)$. Hence, we get $(\tilde x^*,\tilde y^*)\in \Phi(\tilde x,\tilde y)$ such that $\tilde x^*=f(\tilde x,\tilde y)$ and $\tilde y^*\in T(\tilde y)$ satisfies,
	\begin{equation*}
	\langle(f(\tilde x,\tilde y),\tilde y^*), (x,y)-(\tilde x,\tilde y)\rangle\geq 0~\forall\,(x,y)\in X\times K(\tilde x).
	\end{equation*} 
	This affirms the occurrence of projected solution $\tilde x\in X$ for given GNEP $\Gamma$ as per Lemma \ref{GNEPRem}.
\end{proof}
It is noticeable that compactness of the choice sets $X_i$,  is required to prove the above result Theorem \ref{compactresult}. Based on a fixed point result, we ensure the occurrence of the projected solution for GNEP $\Gamma=(X_i,K_i,P_i)_{i\in \Lambda}$ without any boundedness assumption of $X_i$. 
\begin{theorem}
	Assume that for any $i\in \Lambda$:
	\begin{itemize}
		\item[(a)] $X_i$ is a non-empty closed convex subset of $\mathbb{R}^{n_i}$;
		\item[(b)] $K_i:X\rightrightarrows \mathbb{R}^{n_i}$ is closed lower semi-continuous map with non-empty convex values for any $x\in X$;
		\item[(c)] $P_i:\mathbb{R}^{n}\rightrightarrows \mathbb{R}^{n_i}$ admits open graph;
		\item[(d)] $x_i\notin co(P_i(x))$ for any $x\in \prod_{i\in \Lambda} co(\overline{K_i(X)})$.
	\end{itemize}
	Then, $\Gamma=(X_i,K_i,P_i)_{i\in \Lambda}$ admits a projected solution if $K_i(X)$ is relatively compact for any $i$.
\end{theorem}
\begin{proof}
	Define a function $g_i:\mathbb{R}^n\times \mathbb{R}^{n_i}\rightarrow \mathbb{R}_+$ as,
	\begin{equation*}
	g_i(y,x_i)=d((y,x_i),G_i^c),
	\end{equation*}
	where $G_i=\{(x,y_i)\in \mathbb{R}^n\times \mathbb{R}^{n_i} \,|\,y_i\in P_i(x)\}$, $G_i^c=\text{complement of}~G_i=\{(x,y_i)\in \mathbb{R}^n\times \mathbb{R}^{n_i}\,|\,y_i\notin P_i(x)\}$ and $d$ is Euclidean metric. Clearly, $g_i(y,x_i)>0$ if and only if $(y,x_i)\notin G_i^c$ or $x_i\in P_i(y)$. 
	
	Define, $M_i:X\times \mathbb{R}^n\rightarrow \mathbb{R}^{n_i}$ as,
	\begin{equation*}
	M_i(\bar x,\bar y)=\{\bar z_i\in K_i(\bar x)|\,g_i(\bar y,\bar z_i)\geq g_i(\bar y,z_i),\,\forall\,z_i\in K_i(\bar x)\}.
	\end{equation*}
	We claim that $M_i$ is upper semi-continuous map with non-empty and compact values. Since the map $K_i$ is closed, we observe $K_i(x)\subset \overline {K_i(X)}$ is compact for any $x\in X$. Suppose $(x,y)\in X\times \mathbb{R}^{n}$ is arbitrary then $M_i(x,y)\neq \emptyset$ because $g_i(y,.)$ being continuous function attains maxima on the compact set $K_i(x)$. We show that $M_i$ is a closed map. Suppose the sequences $(\bar x_n,\bar y_n)$ and $\bar z_i^n\in M_i(\bar x_n,\bar y_n)$ converges to $(\bar x,\bar y)$ and $\bar z_i$, respectively. We aim to show that $\bar z_i\in M_i(\bar x,\bar y)$. In fact, for any $n\in \mathbb{N}$ we have $\bar z_i^n\in M_i(\bar x_n,\bar y_n)$, which implies
	\begin{equation}\label{res1eq1}
	g_i(\bar y_n,\bar z_i^n) \geq g_i(\bar y_n,z_i)~\text{for all}~z_i\in K_i(\bar x_n).
	\end{equation}
	Furthermore, $\bar z_i\in K_i(\bar x)$ as we know $\bar z_i^n\in K_i(\bar x_n)$ for all $n\in \mathbb{N}$ and $K_i$ is a closed map. Suppose $w_i\in K_i(\bar x)$ is arbitrary. Then, by using the lower semi-continuity of $K_i$ we obtain $w_i^n\in K_i(\bar x_n)$ such that $w_i^n$ converges to $w_i$. Hence, by (\ref{res1eq1}) we have,
	\begin{equation}\label{res1eq2}
	g_i(\bar y_n,\bar z_i^n) \geq g_i(\bar y_n,w_i^n)~\text{for all}~n\in \mathbb{N}.
	\end{equation}
	As we know that $g_i$ is continuous (see Remark \ref{gi}), by taking $n\rightarrow\infty$ in (\ref{res1eq2}) we obtain,
	\begin{equation}\label{res1eq3}
	g_i(\bar y,\bar z_i) \geq g_i(\bar y,w_i).
	\end{equation}
	Since, $w_i\in K_i(\bar x)$ is arbitrary, we observe that (\ref{res1eq3}) holds for all $w_i\in K_i(\bar x)$. Hence, $\bar z_i\in M_i(\bar x,\bar y)$ and $M_i$ is a closed map. This proves $M_i$ is upper semi-continuous with non-empty compact values as $M_i(X\times \mathbb{R}^n)\subset K_i(X)$ is relatively compact \cite[Theorem 17.11]{aliprantis}. 
	
	Let us define the map $M':X\times \mathbb{R}^n\rightrightarrows \mathbb{R}^{n}$ as $M'(x,y)=\prod_{i\in \Lambda} M'_i(x,y)$ where $M'_i:X\times \mathbb{R}^n\rightrightarrows \mathbb{R}^{n_i}$ as $M'_i(x,y)=co(M_i(x,y))$. Then, $M'$ is an upper semi-continuous map with non-empty convex compact values as per \cite[Proposition 2.1]{aussel-cotrina}. Further $M'(X\times \mathbb{R}^n)$ is contained in $Q=\prod_{i\in \Lambda} co(\overline{K_i(X)})$, which is compact.
	
	We define the map $Pr:\mathbb{R}^n\rightrightarrows X$ as,
	\begin{equation*}
	Pr(y)=\{\bar x\in X|\, \norm{y-\bar x}=\inf_{x\in X}\norm{y-x}\}.
	\end{equation*}
Now, $Pr$ is a continuous single-valued map as per Lemma \ref{projection}. Further, we observe that the set $Pr(Q)$ is compact as $Q$ is a compact set.

 Eventually, the map $\Phi: X\times Q\rightrightarrows \mathbb{R}^n\times \mathbb{R}^n$ defined as $\Phi(x,y)=Pr(y)\times M'(x,y)$ becomes upper semi-continuous with non-empty convex compact values. Further, $\Phi(X\times Q)=Pr(Q)\times M'(X\times Q)$ is relatively compact. Hence, we obtain a fixed point $(\tilde x,\tilde y)\in \Phi(\tilde x,\tilde y)$ by virtue of Theorem \ref{kakutani}. 
	
	We claim that $\tilde x$ is a projected solution for the given GNEP. In fact, $\tilde x\in Pr(\tilde y)$ and it remains to show that $\tilde y\in M'(\tilde x)$ satisfies $P_i(\tilde y)\cap K_i(\tilde x)=\emptyset$ for all $i\in \Lambda$. On contrary, suppose $z_i\in P_i(\tilde y)\cap K_i(\tilde x)$ for some $i\in \Lambda$. Then, $(\tilde y,z_i)\notin G_i^c$ and $g_i(y,x)>0$. Thus, for any $w_i\in M_i(\tilde x,\tilde y)$ we have,
	\begin{equation*}
	g_i(\tilde y, w_i)\geq g_i(\tilde y, z_i)>0.
	\end{equation*}
	Hence, $(\tilde y, w_i)\notin G_i^c$ or $w_i\in P_i(\tilde y)$. This proves, $M_i(\tilde x, \tilde y)\subset P_i(\tilde y)$. Consequently, we have $\tilde y_i\in co(M_i(\tilde x, \tilde y))\subset co(P_i(\tilde y))$. This, contradicts our assumption (c). Hence, $P_i(\tilde y)\cap K_i(\tilde x)=\emptyset$ for all $i\in \Lambda$ holds true. 
\end{proof}

\begin{remark}\label{gi}
	One can observe that the function $g_i$ meets continuity. In the view of hypothesis (c), the set $G_i^c$ is closed in $\mathbb{R}^n\times \mathbb{R}^{n_i}$. Hence, for each $(y,x_i)\in \mathbb{R}^n\times \mathbb{R}^{n_i}$ we have atleast one $(y',x'_i)\in G_i^c$ such that,
	\begin{equation}
	d((y,x_i),G_i^c)=d((y,x_i),(y',x'_i)).
	\end{equation}
	By using this fact one can obtain,
	\begin{equation}
	\big|d((z,w_i),G_i^c)-d((y,x_i),G_i^c)\big|\leq d((z,w_i),(y,x_i))
	\end{equation} 
	for any $(z,w_i)$ and $(y,x_i)$ in $\mathbb{R}^n\times \mathbb{R}^{n_i}$. Thus, for any $\epsilon>0$ we can take $\delta=\epsilon$ to show that $\big|g_i(z,w_i)-g_i(y,x_i)\big|<\epsilon$ whenever $d((z,w_i),(y,x_i))<\delta$.
\end{remark}





	
	
	
\end{document}